[1994/12/01]% LaTeX date must December 1994 or later
\documentclass[a4paper]{amsart}%{article}[amssymb]%
\usepackage{amssymb,amsxtra}
\usepackage{amsmath,amsthm,amscd}

\chardef\bslash=`\\ % p. 424, TeXbook
\newtheorem{thm}{Theorem}[section]

\newtheorem{lem}[thm]{Lemma}
\newtheorem{prop}[thm]{Proposition}

\theoremstyle{definition}
\newtheorem{defn}{Definition}[section]

\theoremstyle{remark}

\newcommand{\thmref}[1]{Theorem~\ref{#1}}

\newcommand{\lemref}[1]{Lemma~\ref{#1}}
\newcommand{\propref}[1]{Proposition~\ref{#1}}

\newcommand{\eval}[2][\right]{\relax
  \ifx#1\right\relax \left.\fi#2#1\rvert}

\title{A Hausdorff-Young inequality for measured groupoids}
\author{Patricia Boivin and Jean Renault}
\address{D\'epartment de Math\'ematiques, Universit\'e d'Orl\'eans,
45067 Orl\'eans, France}
\email{Jean.Renault@univ-orleans.fr, boivin.patricia@wanadoo.fr}

\begin{document}
\vskip5mm
\begin{abstract}
The classical Hausdorff-Young inequality for locally compact abelian groups
states that, for $1\le p\le 2$, the $L^{p}$-norm of
a function dominates the $L^{q}$-norm of its Fourier transform, where
$1/p+1/q=1$. By using the theory of non-commutative $L^p$-spaces and by reinterpreting
the Fourier transform, R.~Kunze (1958) [resp. M.~Terp (1980)] extended this
inequality to unimodular [resp. non-unimodular] groups. The analysis of the
$L^p$-spaces of the von Neumann algebra of a measured groupoid provides a further extension of the Hausdorff-Young
inequality to measured groupoids.
\end{abstract} 

\maketitle
\markboth{Patricia Boivin and Jean Renault}
{A Hausdorff-Young inequality for measured groupoids}

\renewcommand{\sectionmark}[1]{}

\section{Introduction.}

The classical Hausdorff-Young inequality for a locally compact abelian group $G$ says that, for $1\le p\le 2$ and for a compactly supported continuous function $f$  on $G$, the $L^{p}$-norm of $f$ dominates the $L^{q}$-norm of its Fourier transform $\hat f$, where $q$ is the conjugate exponent 
($1/p+1/q=1$): $$\|\hat f\|_q\le \|f\|_p.$$
The Fourier transform $\hat f$ is the function on the dual group $\hat G$ defined by
$$\hat f(\chi)=\int f(x)\overline{\chi(x)} dx.$$
The Haar measure of $\hat G$ is normalized so that the Fourier transform extends to an isometry from   $L^2(G)$ onto $L^2(\hat G)$.  Thus, the inequality holds for $p=2$. The proof is usually given as one of the first applications of the complex interpolation method, since the inequality also holds for $p=1$.

This inequality has been generalized to non-abelian locally compact groups. The main difficulty is to give a meaning to the space $L^q(\hat G)$ and to the Fourier transform ${\mathcal F}_p: L^p(G)\rightarrow L^q(\hat G)$ when $G$ is no longer abelian. The theory of operator algebras provides a solution. In the abelian case, the function $\hat f$ can be seen as the multiplication operator $M(\hat f)$ on $L^2(\hat G)$, i.e. $M(\hat f)\hat\xi=\hat f\hat\xi$. Up to unitary equivalence, this is the convolution operator $L(f)$ on $L^2(G)$, i.e. $L(f)\xi=f*\xi$. This still makes sense when $G$ is not abelian, provided one is willing to deal with unbounded operators. When $f$ is continuous and compactly supported, $L(f)$ is a bounded operator on $L^2(G)$. The von Neumann algebra generated by these operators is the group von Neumann algebra $VN(G)$, which will play the r\^ole of the von Neumann algebra $L^\infty (\hat G)$.

The case of unimodular groups was solved by R.~Kunze in \cite{kun:unimodular}. In that case, the von Neumann algebra $M=VN(G)$ has a trace $\tau$ and the non-commutative integration theory of I.~Segal (\cite{seg:integration}) is available. In fact, Kunze defines and studies non-commutative $L^p$-spaces $L^p(M,\tau)$ for a general von Neumann algebra $M$ endowed with a trace $\tau$. Its elements are unbounded operators on the Hilbert space $L^2(M,\tau)$ of the GNS construction affiliated with $M$. He then defines $L^q(\hat G)$ as $L^q(VN(G),\tau)$. For $1\le p\le 2$, the $p$-Fourier transform ${\mathcal F}_p: L^p(G)\rightarrow L^q(\hat G)$ is given by ${\mathcal F}_p(f)=L(f)$ just as above.

One had to wait for the Tomita-Takesaki theory and its developments to tackle the non-unimodular case. Relying on the work of U.~Haagerup (\cite{haa:lp}) on one hand and of A.~Connes (\cite{con:spatial}) and M.~Hilsum (\cite{hil:lp}) on the other on non-commutative $L^p$-spaces, M.~Terp solved the general case of a locally compact group in a paper \cite{ter:modular} unfortunately never published. 
% Her Hausdorff-Young inequality has the  form: 
%$$\|{\mathcal F}_p(f)\|_{L^q(M,\psi_0)}\le \|f\|_p,$$
% where $${\mathcal F}_p(f)=L(f)\Delta^{1/q}$$ and $\Delta$ is the operator of multiplication by the modular function $\delta$ of the group $G$ acting on the Hilbert space $L^2(G)$; $M$ is the von Neumann algebra of the left regular representation of $G$ on $L^2(G)$ and $\psi_0$ is the canonical weight on $M'$.

She also developed in \cite{ter:interpolation}, in the case of a weight instead of a state, the interpolation approach of H.~Kosaki to non-commutative $L^p$-spaces, a framework particularly well-fitted to the generalization of the Hausdorff-Young inequality.

In the meantime, other Hausdorff-Young type inequalities were found. In particular, B.~Russo gave in {\cite{rus:integral} the following Hausdorff-Young inequality for integral operators. Let $X$ be a Borel space endowed with a $\sigma$-finite measure $dx$. Given $f\in L^2(X\times X)$, let $L(f)$ be the integral operator on $L^2(X)$ having $f$ for kernel. Then, for $1\le p\le 2$ and $q$ conjugate exponent of $p$, the following inequality holds:
$$\|L(f)\|_q\le\max{(\|f\|_{p,q}, \|f^*\|_{p,q})},$$
where $\|L(f)\|_q$ is the $q$-Schatten norm of $L(f)$, $\|f\|_{p,q}=\big(\int(\int |f(x,y)|^p dx)^{q/p}dy\big)^{1/q}$ is the mixed norm introduced in \cite {bp:mixed} and $f^*(x,y)=\overline{f(y,x)}$.

When $p=1$, $q=\infty$; the corresponding norm is the usual operator norm and the inequality becomes
$$\|L(f)\|\le\max{(\sup_y(\int |f(x,y)| dx), \sup_x(\int |f(x,y)| dy)})$$
which is a well-known consequence of the Cauchy-Schwarz inequality. When $p=2$, $q=2$; the corresponding norm is the Hilbert-Schmidt norm; both terms on the right hand side are equal by Fubini's theorem; as well known, we have the equality
$$\|L(f)\|_2=\big(\int\int |f(x,y)|^2 dxdy\big)^{1/2}.$$
As in the classical case, the case $1<p<2$ is obtained by complex interpolation. Although the mixed norms
$\|f\|_{p,q}$ are interpolation norms just as the usual $L^p$-norms are, one cannot conclude directly, due to the presence of a $\max$ of these norms.

The similarity between the group case and the case of integral operators is striking. It can still be made stronger by introducing the measured groupoid $G=X\times X$. Then the formulation is almost identical to the group case: the relevant von Neumann algebra is $M=VN(G)$ (it is the algebra of all bounded operators on $L^2(X)$ but acting by left multiplication on $L^2(G)$), $L(f)$ is the left convolution operator on $L^2(G)$ and the Schatten space $S_q$ is the non-commutative $L^q$-space of $M$ endowed with its natural trace. The main difference is that the usual $L^p$-norm of $f$ has been replaced by the maximum of the mixed norms $\|f\|_{p,q}$ and $\|f^*\|_{p,q}$.

Let us fix our notation for groupoids. The unit space of the groupoid $G$ will usually be denoted by $X=G^{(0)}$. We shall use  greek letters like $\gamma, \ldots$ for elements of $G$ and roman letters $x,y,\ldots$ for elements of $G^{(0)}$. The range and source maps will be denoted by $r,s:G\rightarrow G^{(0)}$.  For $x\in G^{(0)}$, we define $G^x=r^{-1}(x)$ and $G_x=s^{-1}(x)$. The inverse of $\gamma\in G$ is denoted by $\gamma^{-1}$. A measured groupoid (see \cite{hah:measured groupoids I}, \cite{hah:measured groupoids II}) consists of a Borel groupoid $G$, a measurable Haar system $\lambda$ made of measures $\lambda^x$ on $G^x$ satisfying the left invariance and the measurability property, and a quasi-invariant measure $\mu$ on $G^{(0)}$.  The quasi-invariance of $\mu$ is expressed by the equivalence of the measures $\nu=\mu\circ\lambda$ and $\nu^{-1}=\mu\circ\tilde \lambda$, where $\tilde\lambda$ is the right invariant Haar system consisting of the measures $\lambda_x=(\lambda^x)^{-1}$ on $G_x$. By definition, the modular function is the Radon-Nikodym derivative $\delta=d\nu/d\nu^{-1}$. We say that the measured groupoid is unimodular if $\delta=1$. To avoid technicalities, we assume that $G$ is a second countable locally compact Hausdorff topological groupoid, that the measures are Radon measures, that $\lambda$ is a continuous Haar system and that the modular function is continuous. Let us give some examples of measured groupoids: locally compact groups (with a chosen left Haar measure), measured spaces and more generally transformation groups $(X,G)$ equipped with a quasi-invariant measure (with the appropriate assumptions to fit our hypotheses). As further examples: the above measured groupoid $X\times X$ and the holonomy groupoid of a foliation equipped with a Lebesgue measure. The mixed norms still make sense in this context. For $f$ non-negative and measurable on $G$ or in $C_c(G)$, one can define
 $$\|f\|_{p,q}=\big(\int (\int |f|^pd\lambda^x)^{q/p} d\mu(x)\big)^{1/q}.$$
 In the case of a group, $\|f\|_{p,q}=\|f\|_p$ while in the case of a space, $\|f\|_{p,q}=\|f\|_q$. On the other hand, we have the regular representation of $G$ on $L^2(G)$ which defines the von Neumann algebra $M=VN(G)$ and provides the operator $L(f)$. Thus, we have all the ingredients to express a Hausdorff inequality for measured groupoids. As we shall see, it will have the following form
 $$\|{\mathcal F}_p(f)\|_{L^q(M)}\le \max{(\|f\|_{p,q}, \|f^*\|_{p,q})},$$
 where $1\le p\le 2$, ${\mathcal F}_p(f)=\Delta^{1/2q}L(f)\Delta^{1/2q}$, $\Delta$ is the operator of multiplication by $\delta$ on $L^2(G)$ and $f^*(\gamma)=\overline{f(\gamma^{-1})}$.

\section{Non-commutative $L^p$-spaces for measured groupoids.}

Let $(G,\lambda,\mu)$ be a measured groupoid. As said above, we assume that $G$ is a second countable locally compact Hausdorff topological groupoid, that the measures are Radon measures, that $\lambda$ is a continuous Haar system and that the modular function $\delta$ is continuous. We denote by $\nu$ the measure
$\nu=\mu\circ\lambda$ on $G$, by $\nu^{-1}$ its image under the inverse map and by $\nu_0$
the symmetric measure $\nu_0=\delta^{-1/2}\nu$. 

It is convenient for our purpose to present the von Neumann algebra $M=VN(G)$ of the measured groupoid as the left algebra of a modular Hilbert algebra (se \cite{tak:tomita}). We endow the space $C_c(G)$ of compactly supported continuous functions on $G$ with the convolution product
$$f*g(\gamma)=\int f(\gamma\eta)g(\eta^{-1})d\lambda^{s(\gamma)}(\eta),$$
the involution
$$f^*(\gamma)=\overline{f(\gamma^{-1})}$$
the inner product
$$(f|g)=\int f(\gamma)\overline{g(\gamma)}d\nu^{-1}(\gamma)$$
and the complex one-parameter group $\Delta(z)$ of automorphisms of $C_c(G)$
$$\Delta(z)f=\delta^z f.$$
The axioms $(I)$ to $(VIII)$ of  \cite[Definition 2.1]{tak:tomita} are readily verified.
The left and right convolution operators $L(f)$ and $R(g)$ are respectively defined by
$$L(f)\xi=f*\xi\qquad R(g)\eta=\eta*g.$$
They extend to bounded linear operators, still denoted by $L(f)$ and $R(g)$, on the completed Hilbert space $H=L^2(G,\nu^{-1})$. The map $L$ is a representation of the $*$-algebra $C_c(G)$. We define
$M=VN(G)$ as the von Neumann algebra generated by $\{L(f), f\in C_c(G)\}$. According to the general theory, the von Neumann algebra generated by $\{R(g), g\in C_c(G)\}$ is the commutant $M'$ of $M$. The polar decomposition $S=J\Delta^{1/2}$ of the closure $S$ of the operator $\xi\mapsto\xi^*$, viewed as an unbounded antilinear operator on $H$ provides the ingredients of the Tomita-Takesaki theory. The modular operator $\Delta$ is the operator of multiplication by $\delta$; it is self-adjoint positive; its domain $Dom(\Delta)$ is the subspace of the $\xi$'s in
$H$ such that $\delta\xi\in H$; it contains $C_c(G)$ as an essential domain. The modular involution $J$ is given by $J\xi=\delta^{1/2}\xi^*$. The canonical weight $\varphi_0$ on $M$ is defined by
$$\varphi_0(L(f)^*L(g))=(g|f)\qquad\hbox{for}\,  f,g\in C_c(G).$$
Another useful formula is 
$$\varphi_0(L(f))=\int f_{|G^{(0)}}d\mu \qquad\hbox{for}\,  f\in C_c(G).$$
The modular group of the weight $\varphi_0$ is given by $\sigma_t(T)=\Delta^{it}T\Delta^{-it}$ for $T\in M$.

The spatial theory (see \cite{con:spatial, hil:lp}) uses the dual weight $\psi_0$ on $M'$, defined by $\psi_0(T)=\varphi_0(JTJ)$. We have
$$\psi_0(R(Jf)^*R(Jg))=(g|f)\qquad\hbox{for}\,  f,g\in C_c(G).$$
In our context, it can be expressed as follows. A vector $\xi\in H$ is called left bounded if the operator $L(\xi)$ defined by $L(\xi)f=\xi*f$ for $f\in C_c(G)$ is bounded. The space of left bounded operators is denoted by ${\mathcal A}_l$. 
One associates to each normal semi-finite weight $\varphi$ on $M$ a self-adjoint positive operator $T$ on $H$ such that
$$\|T^{1/2}\xi\|^2=\varphi(L(\xi)L(\xi)^*)$$
for all $\xi\in {\mathcal A}_l$. The operator $T$ is called the spatial Radon-Nikodym derivative of the weight $\varphi$ and is denoted by $\displaystyle{d\varphi\over d\psi^{(0)}}$. It is called integrable if $\varphi$ is a state. One then defines its integral as $\int T d\psi_0=\varphi(1)$. The correspondence $\varphi\mapsto \displaystyle{d\varphi\over d\psi^{(0)}}$ extends to $\varphi\in M_*$. The space $L^1(M,\psi_0)$ is defined as its image. On the other hand, if $\varphi\in M_*$, $f\mapsto \varphi(L(f))$ is a Radon measure on  $C_c(G)$ which is absolutely continuous with respect to $\nu=\mu\circ\lambda$.  We denote by $\displaystyle{d\varphi\over d\nu_0}$ its usual Radon-Nikodym derivative with respect to the symmetric measure $\nu_0$.  Explicitly, if $\varphi(T)=(T\xi |\eta)$ with $\xi,\eta\in H$, one finds that $\displaystyle{d\varphi\over d\nu_0}$ is the coefficient
$$(\xi',\eta')(\gamma)=(L(\gamma)\xi'\circ s(\gamma)|\eta'\circ r(\gamma))$$
of the left regular representation of $G$ on the field of Hilbert spaces $x\mapsto L^2(G^x,\lambda^x)$, where $\xi'=\delta^{-1/2}\xi$ and $\eta'=\delta^{-1/2}\eta $. We define the space ${\mathcal L}^1(\hat G)$ as the image of the map $\varphi\in M_*\mapsto \displaystyle{d\varphi\over d\nu_0}$. The three spaces $M_*$, $L^1(M,\psi_0)$ and ${\mathcal L}^1(\hat G)$ are isomorphic as linear spaces, and also as Banach spaces with the appropriate norms. When $\varphi(T)=(T\xi |\xi)$ with $\xi\in C_c(G)$, the spatial Radon-Nikodym derivative  $T=\displaystyle{d\varphi\over d\psi^{(0)}}$ is the closure of the operator $\Delta^{1/2}L(\tilde F)\Delta^{1/2}$ defined on $C_c(G)$, where $F=\displaystyle{d\varphi\over d\nu_0}$ and $\tilde F(\gamma)=F(\gamma^{-1})$. When $G$ is a locally compact group, ${\mathcal L}^1(\hat G)$ is the usual Fourier algebra $A(G)$.

The Hilbert space $H=L^2(G,\nu^{-1})$ admits the direct integral decomposition $H=\int^\oplus H_x d\mu_x$, where $H_x=L^2(G_x,\lambda_x)$. A decomposable operator $T=\int^\oplus T_x d\mu_x$ is called $\alpha$-homogeneous if for a.e. $\gamma\in G$,
$$R(\gamma)T_{s(\gamma)}\subset \delta^{-\alpha}(\gamma)T_{r(\gamma)}R(\gamma),$$
where $R(\gamma)$ is the isometry $H_{s(\gamma)}\rightarrow H_{r(\gamma)}$ given by $R(\gamma)\xi(\gamma')=\xi(\gamma'\gamma)$. The elements of $M$ (and more generally the operators affiliated with $M$) are 0-homogeneous. The operator $\Delta$ as well as the elements of $L^1(M,\psi_0)$ are $(-1)$-homogeneous. For $p\in [0,\infty[$, $L^p(M,\psi_0)$ is the set of closed densely defined $(-1/p)$-homogeneous operators $H$ such that $|T|^p$ belongs to $L^1(M,\psi_0)$.  The $p$-norm of $T\in L^p(M,\psi_0)$ is defined as $\|T\|_p=\big(\int |T|^p d\psi_0\big)^{1/p}$. For $p=\infty$, we define  $L^\infty(M,\psi_0)=M$

The well-known isomorphism of $L^2(X\times X)$ onto the space of Hilbert-Schmidt operators on $L^2(X)$ which associates to $f$ the integral operator $L(f)$ with kernel $f$ and the Fourier-Plancherel transform of $L^2(G)$ onto $L^2(\hat G)$, where $G$ is a locally compact abelian group can be viewed as particular cases of the following general result for measured groupoids, which is proved in \cite{boi:these}.

\begin{thm}\label{Plancherel} Let $(G,\lambda,\mu)$ be a measured groupoid. There exists a unique isometry ${\mathcal F}_2: L^2(G,\nu_0)\rightarrow L^2(M,\psi_0)$ such that, for $f\in C_c(G)$, ${\mathcal F}_2(f)$ is the the closure of the operator $\Delta^{1/4}L(f)\Delta^{1/4}$ defined on $C_c(G)$.
\end{thm}

In \cite{ter:interpolation}, M.~Terp defines the non-commutative $L^p$-spaces by the complex interpolation method in a framework perfectly well-fitted to ours. The data consists of a von Neumann algebra $M$ and a (normal, faithful and semi-finite) weight $\varphi$. She defines the ``intersection'' $L$ of $M$ and its predual $M_*$ and then embeds $M$ and $M_*$ into $L^*$, turning  $(M,M_*)$ into a compatible pair of Banach spaces. She also defines (Section 2.3) for all $p\in [1,\infty]$ a linear, norm-decreasing, injective map $\mu_p: L\rightarrow  L^p(M,\psi)$ which has dense range for $p<\infty$ and where $\psi$ is the dual weight of $M'$ in the GNS representation of $(M,\varphi)$. We apply this to $(M=VN(G),\varphi=\varphi_0)$. We observe that $C_c(G)*C_c(G)$ is contained in $L$ and that for $f\in C_c(G)*C_c(G)$, $\mu_p(f)$ is the closure of the operator $\Delta^{1/2p}L(f)\Delta^{1/2p}$ defined on $C_c(G)$. For $p=1$, $f\mapsto \mu_1(\tilde f)$ extends to the isomorphism ${\mathcal L}^1(\hat G)\rightarrow L^1(M,\psi_0)$ described earlier. For $p=2$, $f\mapsto \mu_2(f)$ extends to the above Fourier-Plancherel isometry ${\mathcal F}_2: L^2(G,\nu_0)\rightarrow L^2(M,\psi_0)$. For $p=\infty$, $\mu_\infty=L$ is the left regular representation.

\section{A Hausdorff-Young inequality.}

We use the same notation as in the previous section. Because of our assumptions, $\Delta$ as well as $\Delta^z$, where $z\in{\bf C}$, admit $C_c(G)$ as an essential domain.

\begin{lem} Let $z\in{\bf C}$ and $f\in C_c(G)$. We have the commutation rule
$$L(f)\Delta^{-z}\subset \Delta^{-z} L(\delta^z f) .$$
\end{lem}

\begin{proof} We have seen that $\Delta^z$ implements the automorphism $\Delta(z)$ of the modular Hilbert algebra $C_c(G)$. Therefore, $\Delta^zL(f)\Delta^{-z}\subset L(\delta^z f)$.
\end{proof}

Since $\Delta^z$ as well as $L(f)$ leave $C_c(G)$ invariant, $C_c(G)$ is
contained in the domain of $\Delta^z L(f)\Delta^z$. Therefore 
$\Delta^z L(f)\Delta^z$ admits an adjoint; moreover the equality
$$(\Delta^z L(f)\Delta^z\xi|\eta)=(\xi|\Delta^{\overline z} L(f^*)\Delta^{\overline z}\eta)$$
for $\xi,\eta\in C_c(G)$ shows that this adjoint is an extension of
$\Delta^{\overline z} L(f^*)\Delta^{\overline z}$. In particular,
$\Delta^z L(f)\Delta^z$ is closable.

\begin{defn} Given $\alpha\in{\bf C}$ and $\beta=(1-2\alpha)^{-1}$, we define the
$\beta$-Fourier transform ${\mathcal F}_\beta(f)$ of $f\in C_c(G)$ as the closure of  the operator $\Delta^\alpha L(f)\Delta^\alpha$ defined on $C_c(G)$.
\end{defn}

We have met  particular cases of this definition: when $f$ belongs to $C_c(G)*C_c(G)$ and $\alpha=1/2q$, where $q\in [1,\infty]$, then $\beta=p$ is the conjugate exponent of $q$ and 
${\mathcal F}_p(f)=\mu_q(f)$, where $\mu_q$ is Terqs symmetric embedding recalled earlier. In that case, ${\mathcal F}_p(f)$ belongs to $L^q(M,\psi_0)$.  On the other hand, ${\mathcal F}_2(f)$ belongs to $L^2(M,\psi_0)$ and ${\mathcal F}_1(f)=L(f)$ belongs to $L^\infty(M,\psi_0)=M$ for all $f\in C_c(G)$.

Our goal is to show that, for $1\le p\le 2$, and $f\in C_c(G)$, ${\mathcal F}_p(f)$
belongs to $L^{q}(M,\psi_0)$ where $q$ is the conjugate exponent of $p$.  and satisfies the following Hausdorff-Young
inequality:
$$\|{\mathcal F}_{p}(f)\|_{q}\le \max(\|f\|_{p,q},\|f^*\|_{p,q})$$
where $\|f\|_{p,q}$ denotes the following mixed norm, which generalizes the mixed norms of \cite{bp:mixed}.

\begin{defn} Given $p,q\in[1,\infty]$ and $f\in C_c(G)$, the mixed norm
$\|f\|_{p,q}$ of $f$ is defined as
$$\|f\|_{p,q}=\big(\int \big(\int |f|^p d\lambda^x \big)^{q/p} d\mu(x)\big)^{1/q}$$
with the obvious modifications when $p$ or $q$ is infinite.
\end{defn}

We first verify that this inequality holds when $p=1$ or $p=2$. Indeed, for
$p=1$, this is the well-known inequality
$$\|L(f)\|\le \max(\|f\|_{1,\infty},\|f^*\|_{1,\infty})$$
which is obtained by the Cauchy-Schwarz inequality.
For $p=2$, we have
$$\|{\mathcal F}_2(f)\|_2^2=\int |f|^2 d\nu_0\le (\int |f|^2 d\nu)^{1/2}(\int |f|^2
d\nu^{-1})^{1/2}=\|f\|_{2,2}\|f^*\|_{2,2}$$
again by Cauchy-Schwarz. We want to deduce the general case $1\le p\le 2$ from
these two extremal cases by interpolation. Although the norm $\|.\|_{p,q}$
interpolates between 1 and 2 (in fact between 1 and $\infty$), we cannot
interpolate directly because our formula involves the maximum of two mixed norms.
We proceed exactly as in \cite{ter:modular}. We shall only give the necessary modifications. It relies on the following characterization of $L^p(M,\psi_0)$.

\begin{prop}\label{characterization} \cite[Proposition 2.3]{ter:modular} Let $p\in [1,\infty]$ and define $q$ by $1/p+1/q=1$. Let $T$ be a closed densely defined $(-1/p)$-homogeneous operator on $H$. Then, the following conditions are equivalent:
\begin{enumerate}
\item $T\in  L^p(M,\psi_0)$,
\item there exists a constant $C\ge 0$ such that
$$\forall S\in L^q(M,\psi_0), \forall \xi\in{\mathcal A}_l\cap Dom(T), \forall \eta\in {\mathcal A}_l\cap Dom(S):$$
$$|(T\xi, S\eta)|\le C\|S\|_q\|L(\xi)\|\|L(\eta)\|.$$
\end{enumerate}
If $T\in  L^p(M,\psi_0)$, then $\|T\|_p$ is the smallest $C$ satisfying $(ii)$.
\end{prop}
\medskip
In the sequel, we
assume that
$1<p<2$,  we define $q$ by $1/p+1/q=1$ and we let $f\in C_c(G)$ such that
$\max{(\|f\|_{p,q},\|f^*\|_{p,q})}\le 1$. We let
$M(x,y)=\max{(\|f^x\|_p,\|f_y\|_p)}$. We choose $\epsilon>0$ such that
$\epsilon^{(q-p)}\|f\|_p^p\le 1$ and we set
$M_\epsilon(x,y)=\max{(M(x,y),\epsilon)}$. We define for $z\in{\bf C}$ and
$\gamma\in G$
\begin{eqnarray}
f_z(\gamma)&=&sgn f(\gamma)|f(\gamma)|^{pz} 
M_\epsilon(r(\gamma),s(\gamma))^{q-z(p+q)}
\quad\hbox{if}\quad f(\gamma)\not=0\nonumber \\ 
f_z(\gamma)&=&0\quad\hbox{if}\quad
f(\gamma)=0.\nonumber
\end{eqnarray}
Then $f_z$ belongs to $C_c(G)$. We denote by $\Omega$ the open strip $\{z\in {\bf C}: 1/2\le \Re z\le 1\}$.  We have the following lemma which is proved exactly as \cite[Lemma 4.1]{ter:modular}.

 \begin{lem}\label{bounded} Let $p\in [1,2]$, $\Omega\subset {\bf C}$, $f,f_z\in C_c(G)$ be as above. Let $\xi\in {\mathcal A}_l$. Then
\begin{enumerate}
\item for each $z\in\overline\Omega$, the function 
$$\xi_z={\mathcal F}_{1/z}(f_z)\xi=\Delta^{1-z\over 2}L(f_z)\Delta^{1-z\over 2}\xi$$ is defined and belongs to $H$;
\item the function $z\mapsto \xi_z$ with values in $H$ is bounded on $\overline\Omega$;
\item for each $\eta\in H$, the scalar function $z\mapsto (\xi_z,\eta)$ is continuous on $\overline\Omega$ and analytic on $\Omega$.
\end{enumerate}
\end{lem}

\begin{proof}
The proof given in \cite[Lemma 4.1]{ter:modular} applies with little change. To prove $(i)$ and $(ii)$, one introduces $\eta\in C_c(G)$ and considers the integral
$$H_\eta(z)=\int \xi(\gamma)\overline{{\mathcal F}_{1/{\overline z}}(f_z^*)\eta(\gamma)}d\nu^{-1}(\gamma).$$
By the use of Fubini's theorem, one checks that this is well defined and that
 $$H_\eta(z)=\int \xi_z(\gamma)\overline{\eta(\gamma)} d\nu^{-1}(\gamma).$$
Then, one proves that there exists $C\ge 0$ such that $|H_\eta(z)|\le \|\eta\|_2$. To do that, one applies the Phragmen-Lindel\"of principle. The function $z\to H_\eta(z)$ is bounded continuous on $\overline\Omega$ and analytic on $\Omega$. It suffices to estimate it on the lines $\Re z=1$ and $\Re z=1/2$. 
 
Suppose that $z=1+it$, where $t$ is real. Then, 
$$\xi_z=\Delta^{-it/2}L(f_z)\Delta^{-it/2}\xi$$
belongs to $H$ and we have
$$|H_\eta(z)|\le \|L(f_z)\|\|\xi\|_2\|\eta\|_2.$$
In order to estimate the operator norm of $L(f_z)$, we use the inequality
$$\|L(f_z)\|\le\max{(\|f_z\|_{1,\infty},\|{f_z}^*\|_{1,\infty})}.$$
Let us fix $x\in G^{(0)}$ and estimate 
$$\int |f_z|d\lambda^x=\int |f(\gamma)|^p 
M_\epsilon(x,s(\gamma))^{-p}d\lambda^x(\gamma).$$ 
We write $G^x$ as
the disjoint union $A\cup B$, where $A=\{\gamma\in G^x:
M(x,s(\gamma))\le\epsilon\}$ and $B=G^x\setminus A$. If $\gamma\in A$,
$M_\epsilon(x,s(\gamma))=M(x,s(\gamma))\ge \|f^x\|_p$. Therefore
$$\int_A |f_z|d\lambda^x\le \|f^x\|_p^{-p}\int_A|f(\gamma)|^p 
d\lambda^x(\gamma)\le 1.$$
On the other hand, if $B$ is non-empty, then $\|f^x\|_p<\epsilon$ and for all
$\gamma\in B$,
$M_\epsilon(x,s(\gamma))=\epsilon$. Therefore,
$$\int_B |f_z|d\lambda^x\le \epsilon^{-p}\int_B|f(\gamma)|^p 
d\lambda^x(\gamma)\le \epsilon^{-p}\epsilon^p\le 1.$$
This shows that $\|f_z\|_{1,\infty}\le 2$. One shows similarly
that $\|{f_z}^*\|_{1,\infty}\le 2$. This gives $\|L(f_z)\|\le 2$ and
$$|H_\eta(z)|\le 2\|\xi\|_2\|\eta\|_2.$$

Suppose next that $z={1\over 2}+it$, where $t$ is real. Then, 
$$\xi_z=\Delta^{-it}{\mathcal F}_2(\delta^{it/2}f_z)\xi.$$
This also belongs to $H$ because the domain of operators in $L^2(M,\psi_0)$ contains ${\mathcal A}_l$. We deduce the inequality 
$$|H_\eta(z)|\le \|{\mathcal F}_2(\delta^{it/2}f_z)\xi\|_2\|\eta\|_2.$$
As a corollary of \propref{characterization}, we have the inequality
$$\|{\mathcal F}_2(\delta^{it/2}f_z)\xi\|_2\le\|{\mathcal F}_2(\delta^{it/2}f_z)\|_2\|L(\xi)\|.$$
We use the Plancherel \thmref{Plancherel} to evaluate the norm
$$\|{\mathcal F}_2(\delta^{it/2}f_z)\|_2=\big(\int
|f_z|^2d\nu_0\big)^{1/2}\le \big(\int
|f_z|^2d\nu\int
|f_z|^2d\nu^{-1}\big)^{1/4}.$$
Let us estimate the integral 
$$\int
|f_z|^2d\nu=\int
|f(\gamma)|^p
M_\epsilon(r(\gamma),s(\gamma))^{(q-p)}d(\mu\circ\lambda)(\gamma).$$
We write the domain of summation as the union of three parts
$$A=\{\gamma: \|f^{r(\gamma)}\|_p\ge\max{( \|f_{s(\gamma)}\|_p,\epsilon)}\},$$
$$B=\{\gamma: \|f_{s(\gamma)}\|_p \ge\max{(
\|f^{r(\gamma)}\|_p,\epsilon)}\}$$
and
$$C=\{\gamma: M(r(\gamma),s(\gamma))\le \epsilon\}.$$
If $\gamma\in A$, $M_\epsilon(r(\gamma),s(\gamma))=\|f^{r(\gamma)}\|_p$, therefore
$$\int_A
|f_z|^2d(\mu\circ\lambda)\le\int(\int
|f(\gamma)|^p
d\lambda^x(\gamma))\|f^x\|^{-p}_p\|f^x\|_p^{q}d\mu(x)=\|f\|_{p,q}^{q}\le 1.$$
Similarly,
$$\int_B
|f_z|^2d(\mu\circ\lambda)\le \|f^*\|_{p,q}^{q}\le 1.$$
Finally, if $\gamma\in C$, $M_\epsilon(r(\gamma),s(\gamma))=\epsilon$ and
$$\int_C
|f_z|^2d(\mu\circ\lambda)=\epsilon^{(q-p)}\int_R
|f|^p d(\mu\circ\lambda)\le \epsilon^{(q-p)}\|f\|_p^p.$$
Because of our choice of $\epsilon>0$, this is majorized by 1. The whole integral
is majorized by 3.  One shows in the same fashion that the integral $\int
|f_z|^2d\nu^{-1}$ is also majorized by 3. Therefore,
$\|{\mathcal F}_2(\delta^{it/2}f_z)\|_2\le
\sqrt 3\le 2$. This gives
$$|H_\eta(z)|\le 2\|L(\xi)\|\|\eta\|_2$$ 
on the line $Re z=1/2$.

This proves the desired inequality with $C=\max(\|\xi\|_2, \|L(\xi)\|)$ and $(i)$ and $(ii)$. We have already proved $(iii)$ for $\eta\in C_c(G)$. The general case follows by approximation.
\end{proof}

The proof of the following lemma, which is exactly \cite[Lemma 4.2]{ter:modular}, does not require any modification.

\begin{lem} Let $p\in [1,2]$ and let $q$ be defined by $1/p+1/q=1$. Let $f\in C_c(G)$ be as above and let $S\in L^p(M,\psi_0)$. Then, for all $\xi\in {\mathcal A}_l$ and $\eta\in {\mathcal A}_l\cap Dom(S)$, we have:
$$|({\mathcal F}_p(f)\xi, S\eta)| \le 2 \|S\|_p \|L(\xi)\| \|L(\eta)\|.$$
Note that $\xi\in Dom({\mathcal F}_p(f))$ by \lemref{bounded}.
\end{lem}

\begin{proof} We just recall how the proof goes. We asssume  $\|S\|_p\le 1$. By \cite[Lemma2.5]{ter:modular}, it suffices to consider the case $\eta\in {\mathcal A}_l\cap Dom(|S|^p)$. Let $S=U|S|$ be the polar decomposition of $S$. For $z\in {\overline\Omega}$, one defines
$$\eta_z=U|S|^{pz}\eta.$$
Then the function $z\mapsto\eta_z\in H$ is bounded continuous on $\overline\Omega$ and analytic on $\Omega$. We then define $H(z)=(\xi_z,\eta_{\overline z})$. The scalar function $z\mapsto H(z)$ satisfies the same properties. One proves the estimate
$$|H(z)|\le 2 \|L(\xi)\| \|L(\eta)\|$$
for $z\in \Omega$ by checking it on the lines $\Re z=1$ and $\Re z=1/2$ as in the previous lemma. 

For $z=1+it$, where $t$ is real, one has 
\begin{eqnarray}
|H(z)|&=&(L(\delta^{it/2}f_z)\xi, \Delta^{it}U|S|^{it}|S|^p\eta)\nonumber \\ 
&\le&\|L(\delta^{it/2}f_z)\|\|L(\xi)\|\|L(\eta)\|.\nonumber
\end{eqnarray}
by applying \propref{characterization} with $p=\infty$. Then one uses the estimate
$$\|L(\delta^{it/2}f_z)\|\le 2$$ established in the previous lemma.

For $z={1\over 2}+it$, where $t$ is real, one simply uses Cauchy-Schwarz' inequality
$$|H(z)|\le \|\xi_z\|_2\|\eta_{\overline z}\|_2.$$
We have established earlier the estimate $\|\xi_z\|_2\le 2\|L(\xi)\|$. On the other hand,
$$\|\eta_{\overline z}\|_2=\|U|S|^{p/2}|S|^{-pit}\eta\|_2$$
can be also estimated by \propref{characterization} with $p=2$:
$$\|U|S|^{p/2}|S|^{-pit}\eta\|_2\le\|U|S|^{p/2}\|_2 \|L(|S|^{-pit}\eta)\|\le \|L(\eta)\|.$$
In particular, we have
$$|H(1/p)|=|({\mathcal F}_p(f)\xi, S\eta)|\le 2\|L(\xi)\|\|L(\eta)\|.$$

\end{proof}

One deduces from \propref{characterization} that for $f\in C_c(G)$, ${\mathcal F}_p(f)$ belongs to $L^q(M,\psi_0)$ for all $p\in [1,2]$ where $1/p+1/q=1$ and
$$\|{\mathcal F}_p(f)\|_{L^q(M)}\le 2\max(\|f\|_{p,q},\|f^*\|_{p,q}).$$ This inequality can be improved.
\begin{thm} (Hausdorff-Young) Let $p\in [1,2]$ and $f\in C_c(G)$.  Then ${\mathcal F}_p(f)$ belongs to $L^q(M,\psi_0)$  where $1/p+1/q=1$ and
$$\|{\mathcal F}_p(f)\|_{L^q(M)}\le \max(\|f\|_{p,q},\|f^*\|_{p,q})$$
where $1/p+1/q=1$.
\end{thm}

\begin{proof} We first observe that the mixed norms as well as the $L^p$ norms behave well with respect to products of groupoids. Explictitly, let $(G_i,\lambda_i,\mu_i), i=1,2$ be two measured groupoids and
let $(G=G_1\times G_2,\lambda=\lambda_1\otimes\lambda_2,\mu=\mu_1\otimes\mu_2)$ be their product. Given
$1\le p,q\le\infty$ and
$f_i\in C_c(G)$, we have
\begin{itemize}
\item $\|f_1\otimes f_2\|_{p,q}=\|f_1\|_{p,q}\|\| f_2\|_{p,q}$;
\item ${\mathcal F}_p(f_1\otimes f_2)={\mathcal F}^{(1)}_p(f_1)\otimes {\mathcal F}^{(2)}_p(f_2)$;
\item for $T_i\in L^q(M^{(i)},\psi_0^{(i)})$, $T=T_1\otimes T_2$ belongs to  $L^q(M,\psi_0)$ and $\|T\|_{L^q(M)}=\|T_1\|_{L^q(M^{(1)})}\|T_2\|_{L^q(M^{(2)})}$ where the indices $1,2$ refer to the corresponding groupoids. The  Hilbert space, the von Neumann algebra and the canonical weights of the product groupoid are also tensor products: $H=H^{(1)}\otimes H^{(2)}$, $M=M^{(1)}\otimes M^{(2)}$, $\varphi=\varphi^{(1)}\otimes \varphi^{(2)}$ and $\psi=\psi^{(1)}\otimes \psi^{(2)}$.
\end{itemize}
Given $f\in C_c(G)$, we form 
$$F=(f^*\otimes f)\otimes(f^*\otimes f)\otimes\ldots (f^*\otimes f)$$
where $f^*\otimes f$ is repeated $n$ times. It belongs to $C_c(G^{2n})$. According to the above 
${\mathcal F}_p(F)$ belongs to $L^q(M^{\otimes 2n},\psi^{\otimes 2n}_0)$ and
$$\|{\mathcal F}_p(F)\|_{L^q(M)}\le 2\max(\|F\|_{p,q},\|F^*\|_{p,q}).$$
Since $\|F\|_{p,q}=\|F^*\|_{p,q}=\|f\|_{p,q}^n\|f^*\|_{p,q}^n$ and $\|{\mathcal F}_p(F)\|_{L^q(M^{2n})}=\|{\mathcal F}_p(f)\|_{L^q(M)}$, this gives
$$\|{\mathcal F}_p(f)\|_{L^q(M)}\le 2^{1/2n}\|f\|_{p,q}^{1/2}\|f^*\|_{p,q}^{1/2}.$$
Passing to the limit, one gets:
$$\|{\mathcal F}_p(f)\|_{L^q(M)}\le \|f\|_{p,q}^{1/2}\|f^*\|_{p,q}^{1/2}\le \max(\|f\|_{p,q},\|f^*\|_{p,q}).$$

\end{proof}


\vskip3mm




\begin{thebibliography}{10}

%\bibitem[label]{cle} Auteur, TITRE, editeur, annee
\bibitem{dr:amenable} C. Anatharaman et J. Renault, {\it Amenable groupoids}, 
Contemp. Math., 282, Amer. Math. Soc., Providence, RI, 2001.
\bibitem{bp:mixed} A. Benadek et R. Panzone, {\it The spaces $L^p$ with mixed norm}, Duke Math. J., {\bf 28} (1961), 301-324.
\bibitem{boi:these} P.~Boivin, {\it Espaces L$^p$ non-commutatifs de l'alg\`ebre de von Neumann d'un groupo\"ide mesur\'e}, Th\`ese, Universit\'e  d'Orl\'eans, 2007.
\bibitem{con:spatial} A. Connes, {\it On the spatial theory of von Neumann algebras}, J. Func. Anal., {\bf 35} (1980), 153-164.
\bibitem{hah:measured groupoids I} P. Hahn, {\it Haar measure for measure groupoids}, Trans. Amer. Math. Soc., {\bf 242} (1978).
\bibitem{hah:measured groupoids II} P. Hahn, {\it Regular Representations of measure groupoids}, Trans. Amer. Math. Soc., {\bf 242} (1978).
\bibitem{haa:lp} U. Haagerup, {\it $L^p$-spaces associated with an arbitrary von Neumann algebra},
Alg\`ebres d'op\'erateurs et leurs applications en physique math\'ematique (Proc. Colloq., Marseille, 1977), Colloq. Internat. CNRS, 
{\bf 274} (1979), 175-184.
\bibitem{hil:lp} M. Hilsum, {Les espaces $L^p$ d'une alg\`ebre de von Neumann d\'e finies par la d\'e riv\'ee spatiale}, J. Funct. Anal.,
{\bf 40} (1981), 151-169.
\bibitem{kos:interpolation} H. Kosaki, {\it Applications of the Complex Interpolation Method to a von Neumann
Algebra : Non-commutative $L^p$-Spaces)}, J. of Funct. Anal., {\bf 56} (1984), 29-78. 
\bibitem{kun:unimodular} R. A. Kunze, {$L_p$ Fourier Transforms on locally compact unimodular groups}, Trans. Amer. Math. Soc,
{\bf 89} (1958), 519-540.
\bibitem{mrw:Morita} P. S. Muhly, J. N. Renault and D. P.Williams, {\it Equivalence and isomorphism for groupoid $C^*$-algebras},
J. Operator Theory, {\bf 17} (1987), 3-22.
\bibitem{ren:approach} J.~Renault: {\it A groupoid approach to
$C^*$-algebras}, Lecture Notes in Mathematics, Vol.~{\bf 793}
Springer-Verlag Berlin, Heidelberg, New York (1980).
\bibitem{rus:integral} B. Russo, {\it On the Hausdorff-Young theorem for integral operators}, Pacific J. Math.,
{\bf 68} (1977), 241-253.
\bibitem{seg:integration} I. E. Segal, {\it A non-commutative extension of abstract integration}, Ann. of Math.,
{\bf 57} (1953), n$^o$3.
\bibitem {tak:tomita} M. Takesaki, {\it Tomita's Theory of Modular Hilbert Algebras and its Applications},
Lecture Notes in Math., {\bf 128}, Springer, 1970.
\bibitem {sz:von neumann} S. Str\u atil\u a and L. Zsid\u o, {\it Lectures on von Neumann Algebras}, Abacus Press, Kent, 1979.
\bibitem{ter:modular} M. Terp, {\it $L^p$ Fourier Transformation on non-unimodular locally compact groups}
Kobenhavns Universitet Matematisk Institut, Preprint, 1980.
\bibitem{ter:interpolation}  M. Terp, {\it Interpolation spaces between a von Neumann algebra and its predual}, J. Operator Theory, {\bf 8} (1982), 327-360.


\end{thebibliography}
\end{document}